\documentclass[11pt]{amsart}

\usepackage{amsthm, amsfonts, amssymb,pinlabel,graphicx}
\usepackage[usenames,dvipsnames]{xcolor}
\usepackage{tikz-cd}
\usepackage{mathtools}
\usepackage{pinlabel}
\usepackage{bbm}
\usepackage[pdftex,%
  final,%
  colorlinks=true,%
  linkcolor=NavyBlue,%
  citecolor=NavyBlue,%
  filecolor=NavyBlue,%
  menucolor=NavyBlue,%
  urlcolor=NavyBlue,%
  bookmarks=true,%
  bookmarksdepth=3,%
  bookmarksnumbered=true,%
  bookmarksopen=true,%
  bookmarksopenlevel=2,%
]{hyperref}

\newcommand{\dfn}[1]{{\textbf{#1}}}

\numberwithin{equation}{section}

\DeclareMathOperator{\lk}{lk}
\DeclareMathOperator{\sgn}{Sign}

\newcommand{\zz}{\ensuremath{\mathbb{Z}}}

\newcommand{\torus}{\ensuremath{\mathbb{T}}}

\theoremstyle{plain}
\newtheorem{thm}{Theorem}[section]
\newtheorem{cor}[thm]{Corollary}
\newtheorem{lem}[thm]{Lemma}

\newtheorem{prop}[thm]{Proposition}

\theoremstyle{definition}
\newtheorem{defn}[thm]{Definition}

\newtheorem*{assumption}{Assumption}

\theoremstyle{remark}
\newtheorem{rem}[thm]{Remark}
\newtheorem{ex}[thm]{Example}

\begin{document}

\title[Refined Non-Orientable $4$-genus of Torus Knots]{On a Refinement of the Non-Orientable $4$-genus of Torus Knots}

\author[J.M. Sabloff]{Joshua M. Sabloff} \address{Haverford College, Haverford, PA 19041} \email{jsabloff@haverford.edu} 
\urladdr{\url{https://jsabloff.sites.haverford.edu/}}

\date{\today}
	
\begin{abstract}
	In formulating a non-orientable analogue of the Milnor Conjecture on the $4$-genus of torus knots, Batson developed an elegant construction that produces a smooth non-orientable spanning surface in $B^4$ for a given torus knot in $S^3$.  While Lobb showed that Batson's surfaces do not always minimize the non-orientable $4$-genus, we prove that they always do minimize among surfaces that share their normal Euler number.  We also completely determine the possible pairs of normal Euler number and first Betti number for non-orientable surfaces whose boundary lies in a class of torus knots for which Batson's surfaces are non-orientable $4$-genus minimizers.
\end{abstract}

\maketitle

\section{Introduction}
\label{sec:intro}

The orientable $4$-genus of a knot in $S^3$ and its significant role in low-dimensional topology have been extensively studied, but comparatively little is known about its non-orientable counterpart.  For example, an important milestone in the study of the smooth orientable $4$-genus was Kronheimer and Mrowka's solution of the Milnor Conjecture, which showed that the $3$- and $4$-genera of torus knots coincide \cite{km:surfaces1, km:surfaces2}; see also \cite{rasmussen:slice}. On the other hand, the non-orientable $4$-genus of torus knots is still not known in general, though the goal of this paper is to increase our understanding.

To specify the problem more precisely, say that, for a knot $K \subset S^3$, a properly embedded smooth (non-orientable) surface  $F \subset B^4$ with $\partial F = K$ is a \dfn{(non-orientable) filling of $K$}. The \dfn{non-orientable genus of $K$} is
\[ \gamma_4(K) = \min \{b_1(F) \,:\, F \text{ is a non-orientable filling of } K \}.\] 
Lower bounds on the non-orientable $4$-genus have proven to be more subtle to construct than those on the orientable $4$-genus.  In this paper, we will rely on lower bounds derived from the signature \cite{gl:signature} and on the $\upsilon$ invariant from knot Floer homology \cite{oss:unoriented}; see Section~\ref{sec:lower-bounds} for details.  See \cite{gl:non-ori-4-genus, my:4-genus, viro:positioning, yasuhara:connecting} for other efforts based on classical tools and \cite{allen:geography, ballinger:concordance-kh, batson:non-ori-slice, bkst:non-ori-genus, ds:cs-clasp, gl:non-ori-4-genus, gm:non-ori-genus} for more recent developments based on the Heegaard Floer package, gauge theory, or Khovanov homology.

To approach the question of the non-orientable $4$-genus of torus knots, Batson \cite{batson:non-ori-slice} proposed a construction of minimal-genus fillings using \dfn{pinch moves}, which, briefly, consist of attaching a non-orientable band between two adjacent strands of a torus knot lying on a standard torus.  The result of a pinch move is another torus knot, which, as we shall detail in Section~\ref{ssec:pinch-construction}, is strictly simpler. Thus, repeatedly pinching a torus knot $T(p,q)$ yields an unknot, and attaching the non-orientable bands to a disk bounded by the unknot yields a \dfn{pinch surface} $F(p,q)$.  Batson conjectured that pinch surfaces realize $\gamma_4(T(p,q))$. Jabuka and Van Cott \cite[Corollary 1.12]{jvc:non-ori-milnor} showed that Batson's conjecture holds for the infinite family of torus knots for which the $\upsilon$ bound is sharp (cf.\ Corollary~\ref{cor:oss}). We shall term such knots \dfn{JVC knots}; see Definition~\ref{defn:jvc-knot} for the precise characterization. On the other hand, Lobb \cite{lobb:counterex} showed that, for $T(4,9)$, the pinch surface does not realize the non-orientable $4$-genus; see \cite{bkst:non-ori-genus, longo:counterex} for further counterexamples. 

To explore Batson's conjecture and pinch surfaces more deeply, recall that, in contrast to the orientable case, a non-orientable filling $F$ has an interesting additional invariant:  the Euler class of its normal bundle, termed the \dfn{normal Euler number} $e(F)$. The normal Euler number may be used to refine the non-orientable $4$-genus by defining, for any $e \in 2\zz$, 
\[\gamma_4^e(K) = \min\{b_1(F)\,:\, F \text{ is a non-orientable filling of } K \text{ with } e(F)=e\}.\]
The main result of this paper shows that pinch surfaces do, indeed, minimize genus among fillings that share their normal Euler number.

\begin{thm}[Refined Batson Conjecture] \label{thm:rbc}
	The pinch surface $F(p,q)$ realizes the minimal first Betti number among all non-orientable fillings $F$ of $T(p,q)$ with $e(F) = e(F(p,q))$, i.e.\ $\gamma_4^e(T(p,q)) = b_1(F(p,q))$.
\end{thm}

To give the theorem above more context, we turn to  Allen's work on the ``geography'' of non-orientable fillings \cite{allen:geography}.  For a given knot $K \subset S^3$, Allen asked for the set $R(K)$ of pairs $(e,b) \in \zz^2$ realized by the normal Euler number and non-orientable genus, repectively, of a filling $F$ of $K$.  To state the second main theorem of the paper, we introduce the \dfn{wedge} $W_{(E,B)} \subset \zz^2$, which is defined by
\begin{align*}
W_{(E,B)} = \bigl\{(e,b) \in 2\zz \times \zz \,:\, &\frac{1}{2}|e-E| \leq b-B, \\ & e \equiv 2b \mod 4 \bigr\}.
\end{align*}
For a non-orientable filling $F$, we use the shorthand $W_F = W_{(e(F), b_1(F))}$.

\begin{thm} \label{thm:jvc-geography}
	For any JVC knot $T(p,q)$, the realizable set $R(T(p,q))$ coincides with $W_{F(p,q)}$.  
\end{thm}

The theorem is illustrated in Figure~\ref{fig:jvc-geography}. As a corollary, JVC knots satisfy Allen's conjecture that torus knots realize their non-orientable $4$-genus at a \emph{unique} normal Euler number \cite[Conjecture 1.6]{allen:geography}.

\begin{figure}
\labellist
\small\hair 2pt
 \pinlabel {$b$} [b] at 239 154
 \pinlabel {$e$} [l] at 280 18
 \pinlabel {$(e(F), b_1(F))$} [tl] at 80 32
\endlabellist
\includegraphics{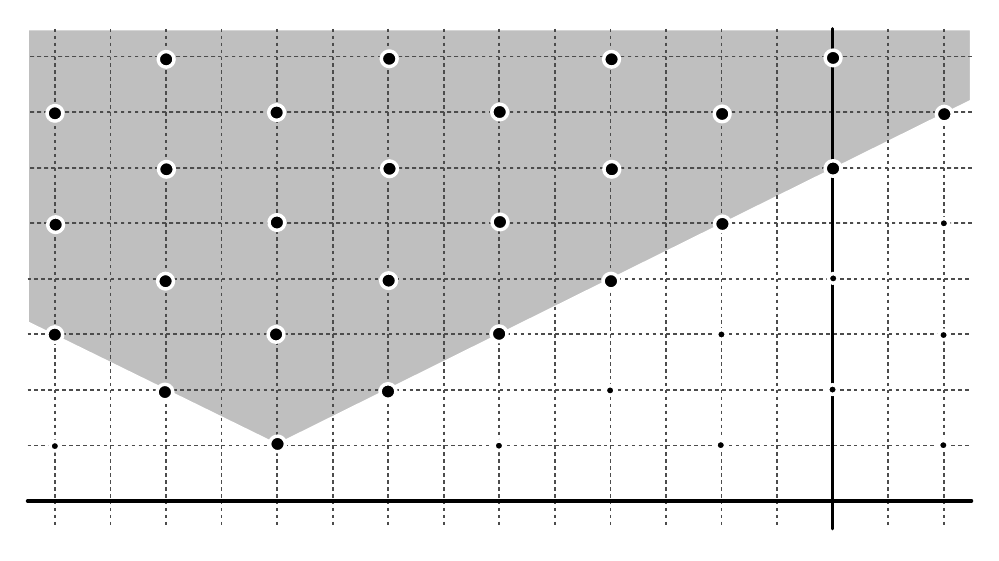}
\caption{The set of realizable pairs $(e(F), b_1(F))$ for fillings $F$ of a JVC knot $T(p,q)$ lie in the shaded wedge $W_{F(p,q)}$. The $T(3,4)$ knot is illustrated here.}
\label{fig:jvc-geography}
\end{figure}

The remainder of the paper is organized as follows.  In Section~\ref{sec:pinch}, we recall the construction of pinch surfaces and compute their normal Euler numbers. We review several lower bounds on the non-orientable genus and define the class of JVC torus knots in Section~\ref{sec:lower-bounds}.  Finally, in Sections~\ref{sec:proof} and \ref{sec:geography}, we prove Theorems~\ref{thm:rbc} and \ref{thm:jvc-geography}, respectively.

\subsection*{Acknowledgements}

The investigations leading to the main result grew out of discussions with Braeden Reinoso about when non-orientable Lagrangian fillings of Legendrian knots realize $\gamma_4^e$.  The author thanks Slaven Jabuka and Cornelia Van Cott for encouragement and generative conversations that helped to contextualize the content of the paper, and the referee for comments that improved the exposition.  The author further thanks the Institute for Advanced Study for hosting him during the preparation of the paper; in particular, this material is based upon work supported by the Institute for Advanced Study.

\section{Pinch Surfaces}
\label{sec:pinch}

\subsection{Construction of Pinch Surfaces}
\label{ssec:pinch-construction}

As described in the introduction, a pinch surface $F(p,q)$ for a torus knot $T(p,q)$ is, roughly speaking, the result of attaching a specific sequence of non-orientable bands to $T(p,q)$ to produce an unknot, and then capping the unknot off with a disk.  We will make this construction more precise using ideas from \cite{batson:non-ori-slice} and \cite{jvc:non-ori-milnor}, adapting notation to better dovetail with later arguments.

We start by defining a non-orientable band cobordism. An embedding $h$ of the square $S=[0,1]^2$ into $S^3$ is an \dfn{non-orientable band} for a knot $K$ if $K \cap h(S) = h([0,1] \times \{0,1\})$, the orientation of $K$ agrees with that of $h([0,1] \times \{0\})$, and the orientation of $K$ disagrees with that of $h([0,1] \times \{1\})$ (or vice versa). A \dfn{non-orientable band move} on $K$ along $h$ is the result of the surgery that replaces the arcs $K \cap h(S)$ with the image of $\{0,1\} \times [0,1]$; we denote the result by $K \# h$.  A \dfn{non-orientable band cobordism} for $K$ is a properly embedded surface in $S^3 \times [0,1]$ with boundary given by $K \subset S^3 \times \{1\}$ and $K \# h \subset S^3 \times \{0\}$ constructed by attaching a $1$-handle to $K \times \{1\}$ along $K \cap h$ with framing specified by $h$.  The direction of the cobordism has been chosen so that successive non-orientable band moves produce a surface in $B^4$.

Next, we specialize to torus knots $T(p,q)$, which we understand to lie on a standard torus $\torus$ in $S^3$. For the rest of the paper, we make the following assumption about the indices $p$ and $q$:

\begin{assumption}
Assume that $p$ and $q$ are positive and relatively prime, and that $q$ is odd with $p > q$ if $p$ is also odd, reversing the roles of $p$ and $q$ if necessary if this were not the case.
\end{assumption}

We visualize $T(p,q)$ as a collection of parallel lines with slope $p/q$ on a unit square with opposite sides identified.  Up to isotopy, there is a unique choice of non-orientable band $h \subset \torus$ for $T(p,q)$.  We call a non-orientable band move on $T(p,q)$ along $h$ a \dfn{pinch move}; see Figure~\ref{fig:pinch-move}. Note that $T(p,q) \# h$ also lies on $\torus$, and hence is, itself, a torus knot $T(r,s)$. We use the term \dfn{pinch cobordism} for the cobordism between $T(r,s) \times \{0\}$ and $T(p,q) \times \{1\}$ resulting from a concatenation of the trace of a ``straightening'' isotopy from $T(r,s)$ to $T(p,q) \# h$ and the non-orientable band cobordism coming from the pinch move.

\begin{figure}
\labellist
\hair 2pt
 \pinlabel {$=$} [ ] at 214 46
\endlabellist
\includegraphics{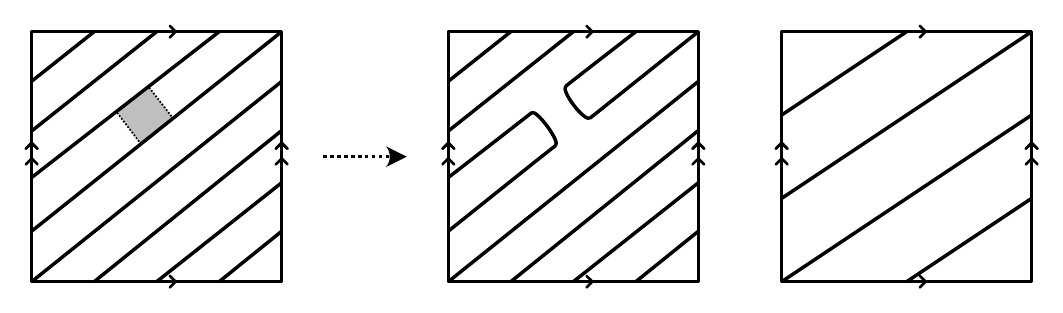}
\caption{A non-orientable band move applied to $T(p,q)$ yields $T(r,s)$.}
\label{fig:pinch-move}
\end{figure}

\begin{lem}[Lemma 2.1 of \cite{jvc:non-ori-milnor}] \label{lem:pinch-move}
	The result of applying a pinch move to $T(p,q)$ is the torus knot $T(p-2t, q-2u)$, where $t$ and $u$ are the integers uniquely determined by the requirements that
	\begin{align*}
		t \equiv -q^{-1} &\mod p \quad \text{and} \quad 0 \leq t \leq p-1, \\
		u \equiv p^{-1} & \mod q  \quad \text{and} \quad 0 \leq u \leq q-1.
	\end{align*}
\end{lem}

The \dfn{sign} of a pinch move is $\epsilon = \sgn(q-2u)$; by \cite[Lemma 2.3]{jvc:non-ori-milnor}, the signs of $p-2t$ and $q-2u$ agree when both quantities are nonzero.  If the sign of a pinch move is negative, then replace $T(r,s)$ by $T(-r,-s)$ to maintain the assumption on the positivity of the indices. Denote a pinch move with sign $\epsilon$ by $T(p,q) \xrightarrow{\epsilon} T(r,s)$.  Lemma~\ref{lem:pinch-move} implies that $0 \leq |r| <p$ and $0 \leq |s| < q$.  Hence, $T(p,q)$ must reduce to the unknot $U$ after some $n$ pinch moves, which we keep track of using the following notation:
\begin{equation} \label{eq:pinch} 
T(p,q) = T(p_n,q_n) \xrightarrow{\epsilon_n} T(p_{n-1},q_{n-1}) \xrightarrow{\epsilon_{n-1}} \cdots \xrightarrow{\epsilon_1} T(p_0,q_0) = U.
\end{equation}
By \cite[Theorem 2.7]{jvc:non-ori-milnor}, we have $q_0 = 1$.

Finally, the \dfn{pinch surface} $F(p,q) \subset B^4$ of $T(p,q)$ is the non-orientable filling of $T(p,q)$ constructed from concatenating the minimal number of pinch cobordisms necessary to transform $T(p,q)$ into an unknot $T(p_0,1)$, and then gluing a disk to the unknot $T(p_0,1)$.  

\subsection{Normal Euler Numbers of Pinch Surfaces}
\label{ssec:normal-euler}

A key ingredient in understanding the refined Batson Conjecture is a computation of the normal Euler number of a pinch surface.  We begin by recalling a definition of the normal Euler number of a closed surface in a $4$-manifold, then proceed to adjust it to a properly embedded surface. See \cite{gl:signature} as well as \cite{batson:non-ori-slice, oss:unoriented} for further discussion.  The actual computation of the normal Euler number of a pinch surface, which will be accomplished in Lemma~\ref{lem:e}, will use Lemma~\ref{lem:oss-e} rather than the details of the definition.  

As a first step, let $F$ be a smooth closed surface in an oriented $4$-manifold $M$; the manifold $M$ will be either $B^4$ or $S^3 \times [0,1]$ in this paper.  To concretely compute the Euler number of the normal bundle to $F$, let $F'$ be a small transverse pushoff of $F$, i.e.\ a section of the normal bundle.  At each point $x \in F \cap F'$, find compatible orientations of $T_x F$ and $T_x F'$, and use those orientations and an ambient orientation of $M$ to assign a sign to each such $x$.  The \dfn{normal Euler number} $e(F)$ is defined to be the sum of the signs of the intersection points $x \in F \cap F'$.  

We next extend this definition to a filling $F \subset B^4$ of a knot $K \subset S^3$.  In fact, it will prove useful to go further and define the normal Euler number for a non-orientable cobordism $F$ in $S^3 \times [0,1]$ between knots $K_0 \subset S^3 \times \{0\}$ and $K_1 \subset S^3 \times \{1\}$; a filling of $K$ may be regarded as a cobordism from $K_0 = \emptyset$ to $K_1 =K$.  To define the normal Euler number in this setting, we cap off $F$ by Seifert surfaces of the $K_i$  in $S^3 \times \{i\}$ to form the closed surface $\hat{F}$, and we take the \dfn{normal Euler number $e(F)$} to be the normal Euler number of the closed surface $\hat{F}$.  This definition is equivalent to constructing a transverse pushoff $F'$ of $F$ so that the push-offs $K'_i$ at the ends realize the Seifert framings of $K_i$, and then summing signed intersections of $F$ and $F'$. 

Some facts about the normal Euler number $e(F)$ to keep in mind are that it vanishes for orientable surfaces, that it is always even, that  it satisfies 
\begin{equation}
	e(F) \equiv 2b_1(F) \mod 4
\end{equation}
(see \cite{Massey:Euler}), and that it is additive under concatenation of cobordisms.  This last fact, combined with the computation of the normal Euler number of a band cobordism in the lemma below, may be used to compute the normal Euler number of a surface constructed via a sequence of band moves.

\begin{lem}[Lemma 4.2 of \cite{oss:unoriented}] \label{lem:oss-e}
	Suppose $h$ is a non-orientable band for a knot $K$.  Choose a section $s$ of the normal bundle of $h$ and framings $\lambda$ of $K$ and $\lambda_h$ of $K \# h$ that agree with $s$ along $K \cap h(S)$.  The normal Euler number of the non-orientable band cobordism $F$ is then computed by
	\[e(F) = \lk(K \# h,\lambda_h) - \lk(K, \lambda).\]
\end{lem}

The following lemma, which computes the normal Euler number of a pinch surface, is the main technical contribution in this paper.

\begin{lem} \label{lem:e}
	If $T(p,q)$ is transformed to $T(p_0,1)$ in the construction of the pinch surface $F(p,q)$, then
	\[e(F(p,q)) = p_0 - pq.\]
\end{lem}

\begin{proof}
	Suppose that the pinch surface from $T(p,q)$ to $T(p_0,1)$ is constructed from $n$ pinches as in Equation~\eqref{eq:pinch}.  We claim that the normal Euler number of the non-orientable band cobordism $F_k$ for $T(p_k, q_k) \to T(p_{k-1},q_{k-1})$ is $p_{k-1}q_{k-1} - p_k q_k$.  The lemma then follows from the claim, the additivity of the normal Euler number, and the fact that the normal Euler number of the final capping disk vanishes.

To prove the claim, we use Lemma~\ref{lem:oss-e} with $K = T(p_k,q_k)$ and $K \# h = T(p_{k-1},q_{k-1})$.  Let the section $s$ and the framings $\lambda$ and $\lambda_h$ all come from a normal vector field to $\torus$; this choice clearly satisfies the hypothesis of the lemma.  Isotope the knot $\lambda$ to another copy of $T(p_k,q_k)$ on $\torus$ by rotating by $\frac{\pi}{2}$ in the direction normal to $K$.  Thus, we see that $\lk(K, \lambda)$ is the same as the linking number between the two components of the torus link $T(2p_k, 2q_k)$.  It is straightforward to check that this linking number is $p_kq_k$.  A similar computation for $K \# h$ and $\lambda_h$ holds, and hence Lemma~\ref{lem:oss-e} implies that
\[
	e(F_k) = \lk(K \# h,\lambda_h) - \lk(K, \lambda) = p_{k-1}q_{k-1} - p_k q_k,
\]
as required.
\end{proof}

\section{Lower Bounds on the Non-Orientable Genus}
\label{sec:lower-bounds}

Of particular importance for this paper are two lower bounds that involve both the non-orientable $4$-genus and the normal Euler number. The first bound relies on the signature.

\begin{thm}[\cite{gl:signature}] \label{thm:gl}
	For any filling $F$ of a knot $K$, we have $\left| \sigma(K) - e(F)/2 \right| \leq b_1(F)$.  In particular, we have
	\[\left| \sigma(K) - \frac{e}{2} \right| \leq \gamma_4^e(K).\]
\end{thm}

See \cite[\S2.3]{allen:geography} or \cite[Theorem 1.1]{yasuhara:connecting} for further discussion of this theorem.

The second bound uses the $\upsilon$ invariant, which is derived from an unoriented version of knot Floer homology.

\begin{thm}[Theorem 1.1 of \cite{oss:unoriented}] \label{thm:oss}
	For any filling $F$ of a knot $K$, we have $\left| 2\upsilon(K) - e(F)/2 \right| \leq b_1(F)$. In particular, we have
	\[\left| 2\upsilon(K) - \frac{e}{2} \right| \leq  \gamma_4^e(K).\]
\end{thm}

To apply Theorem~\ref{thm:oss} to the proof of Theorem~\ref{thm:rbc}, we will use Jabuka and Van Cott's computation of the $\upsilon$ invariant for a torus knot in terms of its pinching sequence \eqref{eq:pinch}.

\begin{thm}[Theorem 1.14 of \cite{jvc:non-ori-milnor}] \label{thm:jvc}
 If $T(p,q)$ is transformed to $T(p_0,1)$ in the construction of the pinch surface, then
	\[2\upsilon(T(p,q)) = b_1(F(p,q)) + \frac{1}{2}(p_0 - pq).\]
\end{thm}

Combining Theorem~\ref{thm:jvc} with Lemma~\ref{lem:e} yields a relationship between $\upsilon$ of a torus knot and the normal Euler number of its pinch surface.

\begin{cor} \label{cor:upsilon-e}
	Under the conditons of Theorem~\ref{thm:jvc}, we obtain:
	\begin{enumerate}
	\item $2\upsilon(T(p,q)) = b_1(F(p,q)) + \frac{1}{2}e(F(p,q))$ and hence
	\item $2\upsilon(T(p,q)) \geq \frac{1}{2}e(F(p,q))$.
	\end{enumerate}
\end{cor}

Returning to the original Theorems~\ref{thm:gl} and \ref{thm:oss}, we may combine them obtain a bound on the non-orientable $4$-genus independent of the normal Euler number.

\begin{cor}[Theorem 1.2 of \cite{oss:unoriented}]
\label{cor:oss}
	For any knot $K \subset S^3$, we have
	\[\left| \upsilon(K) - \frac{\sigma(K)}{2} \right| \leq \gamma_4(K).\]
\end{cor}

Jabuka and Van Cott were able to determine the sign of the quantity on the left of the inequality above for torus knots.

\begin{prop}[Corollary 1.10 of \cite{jvc:non-ori-milnor}] \label{prop:upsilon-sigma}
	For any torus knot $T(p,q)$, we have $2\upsilon(T(p,q)) \geq \sigma(T(p,q))$. 
\end{prop} 

Jabuka and Van Cott also completely characterized which torus knots have a pinch surface that realizes the bound in Corollary~\ref{cor:oss}. As the precise identification of the conditions on the pairs $(p,q)$ that determine these knots is not necessary for this paper, we leave the technicalities for Theorem 1.11 in \cite{jvc:non-ori-milnor}.  Nevertheless, we make the following definition.

\begin{defn} \label{defn:jvc-knot}
	A torus knot $T(p,q)$ is a \dfn{JVC knot} if its pinch surface realizes the bound in Corollary~\ref{cor:oss}, i.e. if 
	\[2 \upsilon(T(p,q)) - \sigma(T(p,q)) = 2b_1(F(p,q)).\]
\end{defn}

\begin{ex}
	As noted in Example 1.13 of \cite{jvc:non-ori-milnor}, the knots $T(2k,2k-1)$ are JVC knots.
\end{ex}

\begin{ex}
	The computations in the proof of Theorem 1.3 in \cite[\S5]{allen:geography} show that $T(3,k)$ is a JVC knot if and only if $k$ is equal to $4$ or $5$ modulo $6$.
\end{ex}

\begin{ex}
	Example 5.3 of \cite{jvc:3vs4}, together with Theorem 1.11(b) of \cite{jvc:non-ori-milnor}, show that $T(km+1,m)$ with $k \geq 1$, $m>1$, and both $m$ and $k$ odd are JVC knots.
\end{ex}

\section{Proof of the Refined Batson Conjecture}
\label{sec:proof}

We now have all of the ingredients needed for the proof of Theorem~\ref{thm:rbc}.  Given a torus knot $T(p,q)$ and its pinch surface $F(p,q)$, the theorem follows from the following short computation, in which we use the notation $e = e(F(p,q))$ for ease of reading:
\begin{align*}
	b_1(F(p,q)) & \geq \gamma_4^e(T(p,q)) \\
	&\geq 2 \upsilon(T(p,q)) - \frac{1}{2}e && \text{by Theorem~\ref{thm:oss}} \\
	&= b_1(F(p,q)) && \text{by Corollary~\ref{cor:upsilon-e}(1)}
\end{align*}

Hence, we obtain $\gamma_4^e(T(p,q)) = b_1(F(p,q))$, as required.  

\section{The Geography of the JVC Knots}
\label{sec:geography}

In this final section, we contextualize Theorem~\ref{thm:rbc} using Allen's geography question.  In particular, we prove Theorem~\ref{thm:jvc-geography}, which completely determines the geography of the JVC knots.  

We begin by defining, for a given knot $K$, two special wedges, which we illustrate in Figure~\ref{fig:wedges} (see also Figures 11 and 12 in \cite{allen:geography}):
\begin{enumerate}
\item Let $W_\sigma$ be the wedge $W_{(2\sigma(K),0)}$, which contains all pairs $(e,b)$ allowed by Theorem~\ref{thm:gl}.
\item Let $W_\upsilon$ be the wedge $W_{(4\upsilon(K),0)}$, which contains all pairs $(e,b)$ allowed by Theorem~\ref{thm:oss}.
\end{enumerate}
Thus, Theorems~\ref{thm:gl} and \ref{thm:oss} yield the following:

\begin{lem} \label{lem:sigma-upsilon-wedge}
	For any knot $K$, $R(K) \subset W_\sigma \cap W_\upsilon$.
\end{lem}

\begin{figure}
\labellist
\small\hair 2pt
 \pinlabel {$e$} [l] at 154 18
 \pinlabel {$b$} [b] at 134 100
 \pinlabel {$4\upsilon$} [t] at 98 16
 \pinlabel {$2\sigma$} [t] at 63 16
 \pinlabel {$(e(F),b_1(F))$} [l] at 67 54
 \pinlabel {$W_\sigma$} [] at 46 35
 \pinlabel {$W_\upsilon$} [] at 115 35
 \pinlabel {$W_F$} [] at 63 76
 \pinlabel {$e$} [l] at 351 18
 \pinlabel {$b$} [b] at 332 100
 \pinlabel {$4\upsilon$} [t] at 297 16
 \pinlabel {$2\sigma$} [t] at 242 16
 \pinlabel {$(e(F),b_1(F))$} [b] at 270 35
 \pinlabel {$W_\sigma$} [] at 225 35
 \pinlabel {$W_\upsilon$} [] at 314 35
 \pinlabel {$W_F$} [] at 270 70
\endlabellist
\includegraphics{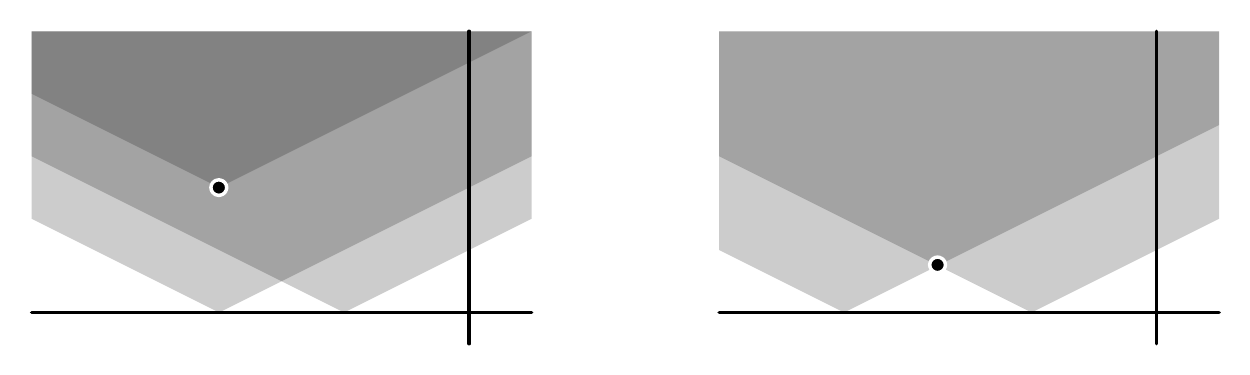}
\caption{The wedges $W_\sigma$, $W_\upsilon$, and $W_F$ for (left) an arbitrary knot $K$ and filling $F$ and (right) for a JVC knot $T(p,q)$ and $F = F(p,q)$.}
\label{fig:wedges}
\end{figure}
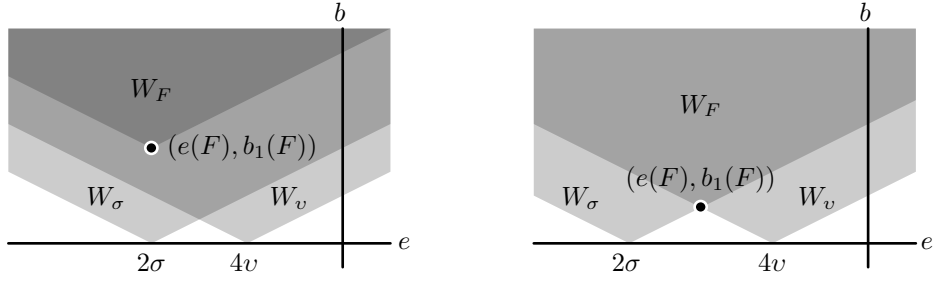

Next, we formalize a fact hinted at in the introduction, namely that every point in the wedge $W_F$ associated to a filling $F$ is realizable.

\begin{lem} \label{lem:F-wedge}
	If $F$ is a non-orientable filling of a knot $K$, then $W_F \subset R(K)$.
\end{lem}

\begin{proof}
	We may modify the filling $F$ to a new filling $F'$ with $b_1(F') = b_1(F)+1$ and $e(F') = e(F) \pm 2$, either by taking a connect sum with a projective plane with normal Euler number $\pm 2$ as in \cite[\S3.1]{allen:geography} or by applying Lemma~\ref{lem:oss-e} as in Figure~\ref{fig:add-2-e}.  Repeating this process appropriately yields fillings for all $(e,b) \in W_F$.
\end{proof}

\begin{figure}
\labellist
\hair 2pt
 \pinlabel {$=$} [ ] at 27 46
 \pinlabel {$=$} [ ] at 202 46
\endlabellist
\includegraphics{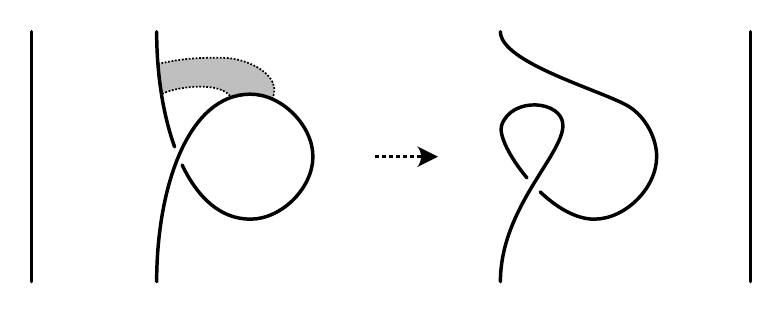}
\caption{Concatenating with the non-orientable band cobordism pictured results in a new filling of a given knot with $b_1$ increased by $1$ and $e$ changed by $2$. To change $e$ by $-2$, switch the crossing.}
\label{fig:add-2-e}
\end{figure}

We now proceed to prove Theorem~\ref{thm:jvc-geography}.  For readability, we denote the torus knot $T(p,q)$ by $K$ and the pinch surface for $T(p,q)$ by $F$.  The proof is visualized in Figure~\ref{fig:wedges}(b), where we see that the pinch surface $F$ lies at the intersection of the wedges $W_\sigma$ and $W_\upsilon$. 

Lemmas~\ref{lem:sigma-upsilon-wedge} and \ref{lem:F-wedge} imply that $W_F \subset R(K) \subset W_\sigma \cap W_\upsilon$.  Thus, it suffices to show that $W_\sigma \cap W_\upsilon \subset W_F$.  Suppose that $(e,b) \in W_\sigma \cap W_\upsilon$.  To show that $(e,b) \in W_F$, we need to prove that $\frac{1}{2}|e-e(F)| \leq b - b_1(F)$ (the parity properties required by the wedge are automatically satisfied).

First, rearrange Corollary~\ref{cor:upsilon-e}(1) to yield
\begin{equation} \label{eq:e-upsilon-b}
	e(F) = 4\upsilon(K) - 2b_1(F).
\end{equation}
Second, combine the sharpness of Corollary~\ref{cor:oss} with Proposition~\ref{prop:upsilon-sigma} to obtain the relation
\begin{equation} \label{eq:upsilon-sigma-b}
	4\upsilon(K) = 2 \sigma(K) + 4b_1(F).
\end{equation}
Plugging Equation~\eqref{eq:upsilon-sigma-b} into Equation~\eqref{eq:e-upsilon-b} produces
\begin{equation} \label{eq:e-sigma-b}
	e(F) = 2 \sigma(K) + 2b_1(F).
\end{equation}

Returning to the goal of proving that $\frac{1}{2}|e-e(F)| \leq b - b_1(F)$, we may compute that
\begin{align*}
	e(F) - e &= 4 \upsilon(K) - e - 2b_1(F) && \text{by \eqref{eq:e-upsilon-b}} \\
	&\leq 2b - 2b_1(F) && \text{since } (e,b) \in W_\upsilon
\end{align*}
and that
\begin{align*}
	e(F) - e &= 2 \sigma(K) - e + 2b_1(F) && \text{by \eqref{eq:e-sigma-b}} \\
	& \geq -2b + 2b_1(F) && \text{since } (e,b) \in W_\sigma
\end{align*}

This completes the proof of Theorem~\ref{thm:jvc-geography}.

\begin{rem}
	The proof of Theorem~\ref{thm:jvc-geography} potentially applies more generally than to JVC knots, as what we really needed was the sharpness of the bounds in Theorem~\ref{thm:oss} and Corollary~\ref{cor:oss} at a given filling $F$ of $K$, along with a guarantee that $\frac{\sigma(K)}{2}$ and $\frac{e(F)}{4}$ lie on the same side of $\upsilon(K)$, as in Proposition~\ref{prop:upsilon-sigma} and Corollary~\ref{cor:upsilon-e}(2), respectively.
\end{rem}

\begin{rem}
	The torus knot $T(4,9)$ investigated in \cite{lobb:counterex} presents a different picture than the JVC knots analyzed above.  Instead of intersecting in the wedge $W_{F(4,9)}$, the wedges $W_\sigma$ and $W_\upsilon$ coincide, and hence give less information than in the JVC knot case.  In fact, it is straightforward, if delicate, to calculate that the pinch surface has $e(F(4,9)) = -36$ and $b_1(F(4,9)) = 2$, while Lobb's minimizing surface $F_{min}$ has $e(F_{min}) = -34$ and $b_1(F_{min}) = 1$. While we already knew from Theorem~\ref{thm:rbc} and Corollary~\ref{cor:upsilon-e}(2) that the pinch surface $F(4,9)$ lies on the lower left boundary of the wedge $W_\upsilon$, we see that $F_{min}$ does as well; see Figure~\ref{fig:lobb-geography}.
\end{rem}

\begin{figure}
\labellist
\small\hair 2pt
 \pinlabel {$e$} [l] at 158 10
 \pinlabel {$-34$} [t] at 61 4
 \pinlabel {$-36$} [t] at 25 4
 \pinlabel {$-32 = 2\sigma = 4\upsilon$} [tl] at 86 4
 \pinlabel {$F_{min}$} [bl] at 64 28
 \pinlabel {$F(4,9)$} [bl] at 28 46
\endlabellist
\centering
\includegraphics{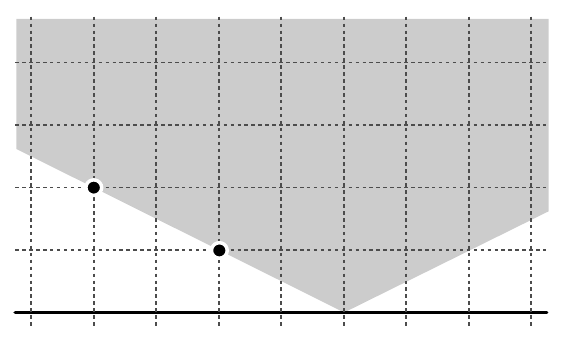}
\caption{Two realizable $(e,b)$ pairs for $T(4,9)$.}
\label{fig:lobb-geography}
\end{figure}
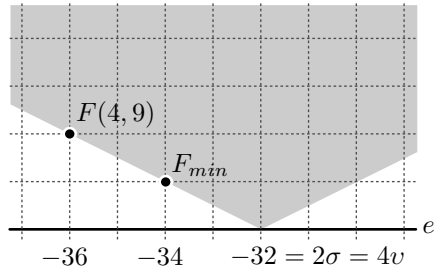


\bibliographystyle{amsplain} 
\bibliography{main}

\end{document}